\documentclass[pdflatex,sn-mathphys-num]{sn-jnl}

\usepackage{graphicx}%
\usepackage{multirow}%
\usepackage{amsmath}%
\usepackage{amssymb}
\usepackage{amsfonts}
\usepackage{amsthm}%
\usepackage{mathrsfs}%
\usepackage[title]{appendix}%
\usepackage{xcolor}%
\usepackage{textcomp}%
\usepackage{manyfoot}%
\usepackage{booktabs}%
\usepackage{algorithm}%
\usepackage{algorithmicx}%
\usepackage{algpseudocode}%
\usepackage{listings}%
\usepackage{dsfont}
\usepackage{amsopn}
\usepackage{mathtools}
\usepackage{mismath}
\usepackage{latexsym}
\usepackage{amssymb}
\usepackage{epsfig,epsf,psfrag}
\usepackage{epstopdf}
\usepackage{array}
\usepackage{bm}
\usepackage{bbm}
\usepackage{diagbox}
\usepackage{color}     
\usepackage{ulem}
\usepackage{stmaryrd}
\usepackage{subcaption}
\usepackage{booktabs}
\usepackage{multirow}
\usepackage{hyperref}
\usepackage{cleveref}
\graphicspath{{figs/}}
\newtheorem{lemma}{Lemma}

\DeclareMathOperator{\Diam}{diam}

\theoremstyle{thmstyleone}%
\newtheorem{theorem}{Theorem}
\theoremstyle{thmstyletwo}%

\theoremstyle{thmstylethree}%

\begin{document}

\title[Article Title]{Multiphysics embedding localized orthogonal decomposition for thermomechanical coupling problems}


\author[1]{\fnm{Yuzhou} \sur{Nan}}\email{2411834@tongji.edu.cn}

\author[2]{\fnm{Yajun} \sur{Wang}}\email{yajunwang@cuhk.edu.hk}

\author[2]{\fnm{Changqing} \sur{Ye}}\email{changqingye@cuhk.edu.hk}

\author[1]{\fnm{Hang} \sur{Qi}}\email{qihang7750@gmail.com}

\author*[1]{\fnm{Xiaofei} \sur{Guan}}\email{guanxf@tongji.edu.cn}

\affil[1]{%
	\orgdiv{School of Mathematical Sciences},%
	\orgname{Tongji University},%
	\orgaddress{\city{Shanghai}, \country{China}}
}

\affil[2]{%
	\orgdiv{Department of Mathematics},%
	\orgname{The Chinese University of Hong Kong},%
	\orgaddress{\city{Hong Kong SAR}, \country{China}}
}

\abstract{Multiscale modeling and analysis of multiphysics coupling processes in highly heterogeneous media present significant challenges. In this paper, we propose a novel multiphysics embedding localized orthogonal decomposition (ME-LOD) method for solving thermomechanical coupling problems, which also provides a systematic approach to address intricate coupling effects in multiphysical systems. Unlike the standard localized orthogonal decomposition (LOD) method that constructs separate multiscale spaces for each physical field, the proposed method features a unified construction for both displacement and temperature. Compared to the standard LOD method, our approach achieves operator stability reconstruction through orthogonalization while preserving computational efficiency. Several numerical experiments demonstrate that the ME-LOD method outperforms the traditional LOD method in accuracy, particularly in cases with significant contrasts in material properties.}




\maketitle

\section{Introduction}
\label{section:intro} \qquad

Heterogeneous media are widely used in practical engineering, with their thermodynamic and mechanical properties spanning multiple temporal and spatial scales \cite{Multiscaling}. These properties are intrinsically linked to the underlying microstructure, composition, and other factors, making their behavior highly complex and difficult to predict \cite{hetemedia}. A key aspect of this complexity is thermomechanical coupling, where thermal and mechanical fields interact bidirectionally. Temperature variations can induce thermal expansion, leading to stress and deformation, while mechanical deformation, in turn, affects heat conduction through changes in material properties or contact conditions. In heterogeneous media, this coupling becomes even more intricate due to spatial variability in material properties, which significantly influences heat transfer, stress distribution, and overall structural response\cite{Huang2021}. These phenomena are particularly relevant in materials science, geophysics, biomedicine, and structural engineering, where accurate modeling is essential to predict performance \cite{14}, ensure structural integrity \cite{11}, and optimize design \cite{topopt,12}. As a result, understanding and efficiently solving thermomechanical coupling problems in heterogeneous media remains a key challenge in computational mechanics and multiphysics simulations.

Based on the first law of thermodynamics and the conservation of momentum, mathematical modeling of thermomechanical processes leads to coupled partial differential equations \cite{existence}, comprising hyperbolic mechanical equations and parabolic heat transfer equations. In the quasi-static case, where the inertia terms of mechanical equations are neglected, the existence and uniqueness of solutions for these coupled PDEs have been analyzed in \cite{N_existence,coupleFEM}. Finite element methods (FEM) have been developed for their numerical solution, with convergence analysis and error estimates provided in \cite{ana_FEM,error_couple}. Extensive research has been conducted on both 2D and 3D linear \cite{2D_FEM,3D_FEM} and nonlinear \cite{nonlinear_FEM} thermomechanical problems, including studies on coupling effects \cite{couple_effects}. 	

The inherent multiscale nature of heterogeneous materials presents significant computational challenges when solving thermomechanical problems using direct numerical methods, which are often prohibitively expensive \cite{ill_condition}. This highlights the need for robust and efficient multiscale methods that employ dimension-reduction techniques while preserving essential microscale information. 

To address these challenges, various multiscale methods have been developed, including constrained macro simulation and the generalized finite element method (GFEM) \cite{GFEM}. Constrained macro simulation derives macroscopic models from microscopic problems using techniques such as homogenization \cite{homogenization, heterogeneousMS}, variational multiscale methods \cite{variationalMS}, and computational homogenization \cite{computationalhomogenization}. Meanwhile, GFEM, first introduced in \cite{GFEM}, constructs specialized finite element spaces to tackle multiscale problems. The multiscale finite element method (MsFEM) \cite{MsFEM} builds basis functions from local microscale problems, with established convergence theories. The generalized multiscale finite element method (GMsFEM) \cite{GMsFEM,GMsFEM2} further improves accuracy by computing multiscale basis functions via local spectral problems. To enhance the robustness and localization properties, the constraint energy minimization generalized multiscale finite element method (CEM-GMsFEM) is proposed, where multiscale basis functions are constructed through auxiliary local spectral problems combined with energy minimization over oversampling domains, leading to exponential decay of basis functions and convergence rates independent of material heterogeneities \cite{CEMGMsFEM,CEMGMsFEM2}. These approaches have been successfully extended to thermomechanical and poroelastic problems.

While GMsFEM constructs multiscale basis functions through local spectral problems, it relies on predefined multiscale spaces and may require a significant amount of computational preprocessing. In contrast, the Local Orthogonal Decomposition (LOD) method \cite{LOD} provides a more flexible and efficient framework by directly modifying standard finite element spaces through localized corrections, making it particularly well-suited for non-periodic heterogeneous media and multiphysics coupling problems \cite{LODGFEM}. Rooted in the variational multiscale method \cite{variationalMS}, LOD enhances the classical finite element framework by modifying the coarse space with locally well-approximated basis functions. This method has proven highly effective in capturing fine-scale features in heterogeneous media while maintaining computational efficiency, especially for elliptic equations. By decomposing the solution space into coarse and fine-scale components, LOD enables the resolution of small-scale heterogeneities without excessive computational cost. Recently, Moritz Hauck and Daniel Peterseim further extended the LOD framework by developing the super-localization method \cite{superLOD}. The super-localization incorporates novel localization techniques that significantly enhance the computational efficiency of the method while preserving its accuracy. This advancement enables more effective handling of complex heterogeneous media and efficiently resolves fine-scale features, further boosting the applicability of LOD for challenging problems.

However, conventional multiscale algorithms face significant challenges when addressing multiphysics coupling problems. For example, extending LOD to more complex thermomechanical problems presents significant challenges \cite{LODGFEM}. Unlike elliptic equations, thermomechanical systems involve coupled parabolic equations, introducing temporal dynamics and intricate spatial interactions that the traditional LOD framework does not readily accommodate. Standard LOD approaches typically handle multiscale problems by constructing basis functions separately for the decoupled elliptic operators of the mechanical and thermal components. But this separation fails to capture the strong coupling effects between thermal and mechanical fields, leading to potential inaccuracies in predicting the behavior of heterogeneous materials under thermomechanical loading.

To overcome these limitations, we propose a novel method, the Multiphysics Embedding  Localized Orthogonal Decomposition (ME-LOD). Compared to conventional LOD methods \cite{LODGFEM,LODGFEM2}, which independently construct multiscale basis functions for the mechanical and thermal parts, the ME-LOD method treats the coupled system as a unified operator. By integrating the governing equations within a unified framework, we develop localized basis functions that inherently encode the interdependencies between thermal and mechanical effects at different scales, thereby providing a robust framework for tackling the multiscale complexities inherent in these systems. By integrating the localization strategy with advancements in multiscale finite element analysis, the ME-LOD method not only retains the advantages of the original method but also enhances its capability to effectively model thermal effects alongside mechanical interactions. By combining localization strategies with advanced techniques, ME-LOD retains the computational efficiency of traditional LOD while significantly improving its ability to model complex thermomechanical interactions. This makes it particularly well-suited for engineering applications where multiscale effects play a crucial role, such as in materials science, geomechanics, and structural analysis.

This paper provides a comprehensive theoretical foundation for ME-LOD, detailing its mathematical formulation, numerical implementation, and practical applicability. The remainder of this paper is structured as follows: Section \ref{sec:notation} introduces the mathematical model and some notations of the thermal-mechanical coupling problem. section \ref{sec:MELOD} starts from the traditional LOD and gives the multiscale space construction and algorithm implementation of the ME-LOD method. Section \cref{sec:error} proves the convergence of the mentioned algorithm. Section \ref{sec:experiments} shows some numerical examples. Section \ref{sec:conclusion} is the summary of this chapter and the follow-up prospects.

\section{General setting and notation}
\label{sec:notation}
In this paper, we investigate a linear thermomechanical problem modeling the quasi-static deformation of heterogeneous media within a domain $\Omega$, where $\Omega \subset \mathbb{R}^d$, $d = 2, 3$, denotes a convex, bounded polygonal or polyhedral domain with Lipschitz continuous boundary $\partial \Omega$. The time interval is given by $\left[0,T\right]$ with $T > 0$. The solid phase is assumed to be a linear elastic material, and the deformation is coupled with the temperature field. Then the problem consists in finding a displacement field $\bm{u} :  \Omega \times [0, T ]  \longrightarrow \mathbb{R}^d$ and a temperature field $\theta :  \Omega \times [0, T ] \longrightarrow \mathbb{R}$ such that

\begin{equation}\label{eq}
	\left\{
	\begin{aligned}
		&-\nabla\cdot\left(\bm{\sigma}(\bm{u})-\alpha\theta\bm{I}\right)=\bm{f},&&\quad\text{ in }\Omega\times(0,T] \\
		&\dot{\theta}-\nabla\cdot\left(\kappa\nabla\theta\right)+\alpha\nabla\cdot\dot{\bm{u}}=g,&&\quad\text{ in }\Omega\times(0,T],
	\end{aligned}
	\right.
\end{equation}
where the initial and boundary conditions are given by
\begin{equation}\label{bound}
	\left\{
	\begin{aligned}
		&\bm{u}=\bm{u}_{D},\quad\mathrm{on}\quad\mathcal{T}_{D}^{u}\times(0,T], 
		&&\quad(\sigma(\bm{u})-\alpha\theta\bm{I})\cdot\bm{n}=\sigma_{N},\quad\mathrm{on}\quad\mathcal{T}_{N}^{u}\times(0,T], \\
		&\theta=\theta_{D},\quad\mathrm{on}\quad\mathcal{T}_{D}^{\theta}\times(0,T],
		&&\quad(\kappa\nabla\theta(x,t))\cdot\bm{n}=0,\quad\mathrm{on }\quad \mathcal{T}_{N}^{\theta}\times(0,T], \\
		&\theta(x,0)=\theta_{0},\quad\mathrm{in}\quad \Omega. 
		&&\quad
	\end{aligned}
	\right.
\end{equation}

Here, $\mathcal{T}_D^u$ and $\mathcal{T}_N^u$ denote two disjoint segments of the boundary such that $\mathcal{T}^u  \coloneqq  \partial \Omega = \mathcal{T}_D^u \cup \mathcal{T}_N^u$. The segments $\mathcal{T}_D^\theta$ and $\mathcal{T}_N^\theta$  are defined similarly. $\bm{u} : \Omega \times [0, T]  \rightarrow \mathbb{R}^d$ denote the displacement field and $\theta : \Omega \times [0, T]  \rightarrow \mathbb{R}$ the temperature. 
The superscript dot represents partial differentiation concerning time $t$, and $\bm{x}$ denotes the space coordinates. $\kappa(\bm{x})$ and $\alpha(\bm{x})$ are the thermal conductivity and expansion coefficients with multiscale characteristics. $\bm{f}$ is the body force and $g$ is the heat source.

For a general thermoelastic material, the constitutive relation of stress and strain is given by 	

$$\sigma(\bm{u})=2\mu\epsilon(\bm{u})+\lambda\nabla\cdot\bm{u}\bm{I},\quad \mathrm{and} \quad\epsilon(\bm{u})=\frac12(\nabla\bm{u}+\nabla\bm{u}^T), $$
where $\lambda, \mu > 0$ are the Lam\'{e} coefficients, which can also be expressed in terms of Young?s modulus $E > 0$ and the Poisson ratio $\nu \in (-1,1/2)$

$$
\mu = \frac{E}{2(1+\nu)}, \quad \lambda = \frac{E\nu}{(1+\nu)(1-2\nu)}.
$$ 
Assume that the material parameters $\lambda, \mu, \kappa, \beta \in L^{\infty}(\Omega)$ satisfied 

$$
\begin{aligned}
	&\begin{aligned}
		& 0<\lambda_{\min } \coloneqq \underset{\bm{x} \in \Omega}{\operatorname{essinf}} \lambda(\bm{x}) \leq \underset{\bm{x} \in \Omega}{\operatorname{esssup}} \lambda(\bm{x}) \coloneqq \lambda_{\max }<\infty ,\\
		& 0<\mu_{\min } \coloneqq \underset{\bm{x} \in \Omega}{\operatorname{essinf}} \mu(\bm{x}) \leq \underset{\bm{x} \in \Omega}{\operatorname{essinf}} \mu(\bm{x})=\mu_{\max }<\infty ,\\
		& 0<\kappa_{\min } \coloneqq \underset{\bm{x} \in \Omega}{\operatorname{essinf}} \kappa(\bm{x}) \leq \underset{\bm{x} \in \Omega}{\operatorname{essinf}} \kappa(\bm{x}) \coloneqq \kappa_{\max }<\infty ,\\
		& 0<\beta_{\min } \coloneqq \underset{\bm{x} \in \Omega}{\operatorname{essinf}} \beta(\bm{x}) \leq \underset{\bm{x} \in \Omega}{\operatorname{esssup}} \beta(\bm{x}) \coloneqq \beta_{\max }<\infty.
	\end{aligned}
\end{aligned}
$$
Let $(\cdot, \cdot ) $ denote the inner product in $L^2(\Omega)$, $\norm{\cdot}$ for the corresponding norm, and : represent the Frobenius inner product.

Moreover, we denote $L^p(0,T;V)$ for the Bochner space with the norm 

$$\begin{aligned}&\norm{\bm{v}}_{L^{p}([0,T];V)}=\left(\int\limits_{0}^{T}\norm{\bm{v}}_{V}^{p}dt\right)^{\frac{1}{p}},\quad1\leq p<\infty,\\&\norm{\bm{v}}_{L^{\infty}([0,T];V)}=\underset{0\leq t\leq T}{\operatorname{essinf}}\norm{\bm{v}}_{V},\end{aligned}$$
where $(V,\norm{\cdot}_V)$ is a Banach space. Also, define $H^1(0,T;V)  \coloneqq  \{\bm{v} \in L^2(0,T;V):\partial_t \in L^2(0,T;V) \}$. To shorten notation, define the spaces for the displacement $\bm{u}$ and temperature $\theta$ by 
\begin{align*}
	&V_u(\Omega) \coloneqq \left\{\bm{v}\mid \bm{v}\in\left[{H}^1(\Omega)\right]^d,\bm{v}=0 ~ \mathrm{on}~ \mathcal{T}_D^u\right\},\\
	&V_\theta(\Omega) \coloneqq \left\{v\mid v\in {H}^1(\Omega),v=0~ \mathrm{on}~\mathcal{T}_D^\theta\right\},
\end{align*}
and 

$$
V(\Omega) \coloneqq V_u(\Omega)\times V_\theta(\Omega)=\{(\bm{v},\theta)\mid \bm{v} \in V_u(\Omega), \theta \in V_\theta(\Omega)\}.$$

Then we obtain the following variational problem: to find weak solutions $\bm{u} \in V_u(\Omega)$ and $\theta \in V_\theta(\Omega)$ such that

\begin{equation}\label{variation}\begin{cases}a(\bm{u},\bm{v}_u)-b(\bm{v}_u,\theta)=\langle\bm{f},\bm{v}_u\rangle,&\forall\bm{v}_u\in V_u(\Omega),\\[2ex]c(\dot{\theta},v_\theta)+d(\theta,v_\theta)+b(\dot{\bm{u}},v_\theta)=\langle g,v_\theta\rangle ,&\forall v_\theta\in V_\theta(\Omega),\end{cases}\end{equation}
where

$$
\begin{aligned} a(\bm{u},\bm{v}_{u})=\int\limits_{\Omega}\sigma(\bm{u}):\epsilon(\bm{v}_{u}),\quad b(\bm{v}_{u},\theta)=\int\limits_{\Omega}\alpha\theta\nabla\cdot\bm{v}_{u},
	\\c(\theta,v_{\theta})=\int\limits_{\Omega}\alpha\theta v_{\theta},\quad d(\theta,v_{\theta})=\int\limits_{\Omega}\kappa\nabla\theta\cdot\nabla v_{\theta}.
\end{aligned}$$
The proof of existence and uniqueness of solutions $\bm{u}$ and $\theta$ to (\ref{variation}) can be found in \cite{coupleFEM}.

To discretize the variational problem (\ref{variation}), let $\mathcal{T}^h$ be a conforming partition for the computational domain $\Omega$ with local grid sizes $h_K \coloneqq  \Diam(K)$ and $h \coloneqq \max_{K \in \mathcal{T}^h}h_K$, where $K \in \mathcal{T}^h$ is a closed subset (of the domain$\Omega$) with a nonempty interior and a piecewise smooth boundary. Here, $\mathcal{T}^h$ is remarked to refer to the $fine$ $ grid$. Next, let $V_{u,h}$ and $V_{\theta,h}$ be the classical $P_1$ finite element spaces with respect to the fine grid $\mathcal{T}^h$, i.e.,

$$\begin{aligned}
	V_{u,h} \coloneqq \{\bm{v}\in (C^{0}(\bar{\Omega}))^d\mid \bm{v}|_{K} \mathrm{is~a~polynomial~of~degree} \leq1 , \  \forall K\in\mathcal{T}^{h}\},
	\\V_{\theta,h} \coloneqq \{v\in C^{0}(\bar{\Omega})\mid v|_{K} \mathrm{is~a~polynomial~of~degree} \leq1, \  \forall K\in\mathcal{T}^h\}.\end{aligned}$$
Denote 

$$V_h \coloneqq V_{u,h}\times V_{\theta,h}=\{ (\bm{v}_1,v_2)\mid \bm{v}_1 \in V_{u,h},\ v_2\in V_{\theta,h}\},$$
and

$$V_H \coloneqq V_{u,H}\times V_{\theta,H}=\{ (\bm{v}_1,v_2)\mid \bm{v}_1 \in V_{u,H},\ v_2\in V_{\theta,H}\}.$$

For time discretization, let $\tau$ be a uniform time step and define $t^n = n\tau $ for $n=0,1,\cdots ,N$ and $T = N\tau$. Then, the finite element method with a backward Euler scheme in time reads, to find $\bm{u}_h \in V_{u,h}$ and $\theta_h \in V_{\theta,h}$ such that 
\begin{equation}\label{eq:discrete}\begin{cases}
		a(\bm{u}_{h}^{n},\bm{v}_u)-b(\bm{v}_u,\theta_{h}^{n})=(f^{n},\bm{v}_u),\quad \forall \bm{v}_u \in V_{u,h},
		\\c(D_{\tau}\bm{u}_{h}^{n},v_\theta)+d(\theta_{h}^{n},v_\theta)+b(D_{\tau}\theta_{h}^{n},v_\theta)=(g^{n},v_\theta),\quad \forall v_\theta \in V_{\theta,h},
\end{cases}\end{equation}
where $D_\tau \theta_h^n  \coloneqq (\theta_h^n-\theta_h^{n-1})/\tau $ and similarly to $D_\tau \bm{u}_h^n$. The right hand sides are evaluated at time $t^n$, that is, $\bm{f}^n  \coloneqq \bm{f}(t^n)$ and $g^n  \coloneqq  g(t^n)$. Given initial data $\theta_h^0$ the system Eq.(\ref{eq:discrete}) is well-proposed \cite{error_couple}. Assume that $\theta_h^0 \in V_{\theta,h}$ is a suitable approximation of $\theta_0$. Note that the first equation of (\ref{variation}) is used to define a consistent initial value $\bm{u}^0  \coloneqq  \bm{u}(0,\cdot) \in V_u$, that is, $\bm{u}^0$ satisfies the equation

$$
a(\bm{u}^0,\bm{v}_u)-b(\bm{v}_u,\theta^0) = (\bm{f}^0,\bm{v}_u), \quad \forall \bm{v}_u \in V_u,
$$
and define $\bm{u}_h^0 \in V_{u,h},$ to be the solution to

\begin{equation}
	a(\bm{u}_h^0,\bm{v}_u) -b(\bm{v}_u,\theta_h^0) = (\bm{f}^0,\bm{v}_u), \quad \forall \bm{v}_u \in V_{u,h}.
\end{equation}
Let $\bm{w}^n = (\bm{u}_h^n, \theta_h^n)$ be the solutions at the $n_{th}$ time  level $t^n $. Then,  following the standard finite element discretization Eq.(\ref{eq:discrete}) can be rewritten as follows

$$A \bm{\bm{w}}^n=B \bm{\bm{w}}^{n-1}+ F^n,$$
where
\begin{equation}A=\begin{bmatrix}A_1&-A_2\\A_3&M^{\prime}+\tau A_4\end{bmatrix},B=\begin{bmatrix}0&0\\A_3&M^{\prime}\end{bmatrix},F^n=\begin{bmatrix}F^n\\\tau G^n\end{bmatrix},
	\label{eq:iterationmatrix}\end{equation}
and 

$$\begin{aligned}&(A_1)_{j_1j_2}=a(\phi_{j_1,u},\ \phi_{j_2,u}), \ (A_2)_{j_1j_2}=b(\phi_{j_1,u},\phi_{j_2,\theta}),\ A_3=A_2^T,\\&(A_4)_{j_1j_2}=d(\phi_{j_1,\theta},\ \phi_{j_2,\theta}), \ (M^{\prime})_{j_1j_2}=c(\phi_{j_1,\theta},\phi_{j_2,\theta}),\\&(F^n)_{j_1}=({\bm{f}}^n,\phi_{j_1,u}),\ (G^n)_{j_2}=(g^n,\phi_{j_2,\theta}),\end{aligned}$$
where $\phi_{j,u}$ are the $j_{th}$ basis function of $V_u$ and similar to $\phi_{j,\theta}$.

Since the classical finite element method approach will serve as a reference solution, this research aims to construct a reduced system based on  Eq.(\ref{eq:discrete}). To this end, a finite-dimensional multiscale space $V_\mathrm{ms} \subseteq \left[H_0^1(\Omega)\right]^d$, whose dimensions are much smaller, will be introduced for approximating the solution on some feasible coarse grid.
\section{Construction of multiscale space}
\label{sec:MELOD}

In this section, we shall introduce a better way to construct the local orthogonal decomposition multiscale spaces than ever before for the multiscale problem under consideration. Subsection \ref{subsec:OD} introduces a quasi-interpolation operator used in the construction of the new basis and defines a modified nodal basis. This basis is then localized in Subsection \ref{subsec:localization}. Subsection \ref{subsec:implementation} introduce the algorithm implementation of the ME-LOD method.

\subsection{Coupled orthogonal decomposition}
\label{subsec:OD}
The classical FEM solution $u_H$ in the coarse space $V_H$ can not accurately approximate $u$. However, it is cheaper to compute than $u_h$ since $\dim(V_H ) \ll \dim(V_h)$. The aim is now to define a new multiscale space $V_\mathrm{ms}$ with the same dimension as the coarse space $V_H$ but with better approximation properties. To this end, the LOD method starts with several macroscopic quantities of interest, which extract the desired information from the exact solution \cite{num_homo}. Let $\mathcal{N}$ denote the set of interior vertices of $\mathcal{T}^H$, with $\mid\mathcal{N}\mid = N^v$. For each vertex $x_m \in \mathcal{N}$, where $m = 1, 2, \dots, N^v$, let $\bm{\lambda}_m \in V_H$ denote the associated nodal basis function (tent function). Correspondingly, let $\bm{\lambda}_{m,u}$ and $\bm{\lambda}_{m,\theta}$ denote the basis functions in $V_{H,u}$ and $V_{H,\theta}$, respectively. The continuous linear functionals are denoted as $q_{j,u} \in \left[V_{u,h}(\Omega)\right]^\ast$ and $q_{j,\theta} \in \left[V_{\theta,h}(\Omega)\right]^\ast,$ $j \in J$, where $J$ is the finite index set with $x_j \in \mathcal{N}$.
Without loss of generality, we assume that these functionals are linearly independent. A canonical choice of functional $q_{j,u}$ and $q_{j,\theta}$ are

$$q_{j,u}  \coloneqq  (\bm{\lambda}_{j,u} , \cdot)_{L^2(\Omega) },\ q_{j,\theta}  \coloneqq  (\bm{\lambda}_{j,\theta} , \cdot)_{L^2(\Omega) }.  $$

Now define the kernels of $q_{j,u}$ and $q_{j,\theta}$ 

$$
V_{\mathrm{f},u} \coloneqq \left\{\bm{v} \in V_{h,u}| q_{j,u} (\bm{v})=0, \forall j \in J\right\}, V_{\mathrm{f},\theta} \coloneqq \left\{v \in V_{h,\theta}| q_{j,\theta} (v)=0 , \forall j \in J\right\}.
$$
The kernels are fine-scale spaces in the sense that they contain all features that are not captured by the (coarse) finite element spaces $V_{H,u}$ and $V_{H,\theta}$. Note that the interpolation leads to the splits $V_{h,u}=V_{H,u} \oplus V_{\mathrm{f},u}$ and $V_{h,\theta}=V_{H,\theta} \oplus V_{\mathrm{f},\theta}$, which means any function $\bm{v}_u \in V_{h,u}$ can be uniquely decomposed as $\bm{v}_u=\bm{v}_{H,u}+\bm{v}_{\mathrm{f},u}$, with $\bm{v}_{H,u} \in V_{H,u}$ and $\bm{v}_{\mathrm{f},u} \in V_{\mathrm{f},u}$, and similarly to $\bm{v}_\theta \in V_{h,\theta}$.

Here, we introduce a Ritz projection onto the fine-scale spaces $\mathcal{C}: V_{h,u} \times V_{h,\theta} \rightarrow V_{\mathrm{f},u} \times V_{\mathrm{f},\theta}$, such that for all $\left(\bm{v}_u, v_\theta\right) \in V_{h,u} \times V_{h,\theta}$, $\mathcal{C}\left(\bm{v}_u, v_\theta\right)=\left(\mathcal{C}_u \bm{v}_u, \mathcal{C}_\theta v_\theta\right)$ fulfills

$$
\begin{array}{ll}
	\left(\sigma\left(\bm{v}_u-\mathcal{C}_u \bm{v}_u\right): \varepsilon\left(\bm{w}_u\right)\right)=0, & \forall \bm{w}_u \in V_{\mathrm{f},u}, \\
	\left(\kappa \nabla\left(v_\theta-\mathcal{C}_\theta v_\theta\right), \nabla w_\theta\right)=0, & \forall w_\theta \in V_{\mathrm{f},\theta} .
\end{array}
$$
Note that this is an uncoupled system and $\mathcal{C}_u$ and $\mathcal{C}_\theta$ are classical Ritz projections.
For any $\left(\bm{v}_u, v_\theta\right) \in V_{h,u} \times V_{h,\theta}$ we have, due to the splits of the spaces $V_{h,u}$ and $V_{h,\theta}$ above, that

$$
\bm{v}_u-\mathcal{C}_u \bm{v}_u=\bm{v}_{H,u}-\mathcal{C}_u \bm{v}_{H,u}, \quad v_\theta-\mathcal{C}_\theta v_\theta=v_{H,\theta}-\mathcal{C}_\theta v_{H,\theta}.
$$
Hence, we can define the multiscale spaces

$$
V_{\mathrm{ms},u} \coloneqq \left\{\bm{v}-\mathcal{C}_u \bm{v}\mid \bm{v} \in V_{H,u}\right\}, \quad V_{\mathrm{ms},\theta} \coloneqq \left\{v-\mathcal{C}_\theta v\mid  v \in V_{H,\theta}\right\} .
$$
Clearly $V_{\mathrm{ms},u} \times V_{\mathrm{ms},\theta}$ has the same dimension as $V_{H,u} \times V_{H,\theta}$. Indeed, with $\bm{\lambda}_{m_1, u}$ denoting the hat function in $V_{H,u}$ at node $x_{m_1}$ and $\bm{\lambda}_{m_2, \theta}$ the hat function in $V_{H,\theta}$ at node $x_{m_2}$, such that a basis for $V_{\mathrm{ms},u} \times V_{\mathrm{ms},\theta}$ is given by

$$
\left\{\left(\bm{\lambda}_{m_1, u}-\mathcal{C}_u \bm{\lambda}_{m_1, u}, 0\right),\left(0, \bm{\lambda}_{m_2, \theta}-\mathcal{C}_\theta \bm{\lambda}_{m_2, \theta}\right)\mid (x_{m_1}, x_{m_2})\in \mathcal{N}\right\}.
$$
We also note that the splits $V_{h,u}=V_{\mathrm{ms},u} \oplus V_{\mathrm{f},u}$ and $V_{h,\theta}=V_{\mathrm{ms},\theta} \oplus V_{\mathrm{f},\theta}$ hold, which fulfill the following orthogonality relation

$$
\begin{aligned}
	\left( \sigma(\bm{v}_u): \varepsilon\left(w_u\right)\right)=0, & \forall \bm{v}_u \in V_{\mathrm{ms},u}, w_u \in V_{\mathrm{f},u}, \\
	\left(\kappa \nabla v_\theta, \nabla w_\theta\right)=0, & \forall v_\theta \in V_{\mathrm{ms},\theta}, w_\theta \in V_{\mathrm{f},\theta}.
\end{aligned}
$$

The core idea of the LOD method is to find a function space orthogonal to the given kernel space $V_\mathrm{f}$ in the sense of $a(\cdot, \cdot)$, which is a sesquilinear form corresponding to a second-order elliptic differential operator. 
In prevailing methodologies, basis function design commonly relies on decoupling techniques, wherein the basis functions are constructed independently of the coupling terms and are therefore incapable of representing coupling effects. Although auxiliary approaches can be employed to supplement coupling information, a more desirable objective is to embed such information directly within the basis functions.

However, the explicit incorporation of coupling terms during basis function construction is generally inadvisable, as the non-regularities they induce render classical analytical techniques unsuitable. To overcome this limitation, we introduce a novel regularization framework that integrates two relaxation coefficients, enabling the systematic inclusion of coupling effects while preserving computational tractability.
Here, we adopt a coupled approach to construct multiscale basis functions. Consider the following static differential operator
\begin{subequations}
	\begin{align}
		L_\gamma(\bm{u},\theta) \coloneqq \left(-\nabla\cdot(\sigma(\bm{u})-\gamma_1\alpha\theta\bm{I}),-\nabla\cdot\left(\kappa\nabla\theta\right)+\gamma_2\alpha\nabla\cdot\bm{u}\right), \label{coperator_gamma}\\
		\intertext{with corresponding variational form}
		l_\gamma\left(\bm{w},\bm{v}\right) = a(\bm{u},\bm{v}_u)-\gamma_1 b(\bm{v}_u,\theta_{h})+d(\theta,v_\theta)+\gamma_2 b(\theta,v_\theta),\label{var4coperator_gamma}
	\end{align}
\end{subequations}
where $\bm{w} \coloneqq (\bm{u},\theta)$, $\bm{v}\coloneqq  (\bm{\bm{v}_u},v_\theta) \in V.$
We define the coupled functional 

$$q_j \coloneqq (\bm{\lambda}_j,\cdot),$$
which leads to the kernel space $V_\mathrm{f} \coloneqq \left\{\bm{v} \in V_{h}\mid q_{j} (\bm{v})=0, \ \forall  j \in J\right\}$. Then the coupled Riesz projection $\mathcal{\tilde{C}}$ can be defined as $\mathcal{\tilde{C}} : V \longrightarrow V_\mathrm{f}$ such that 

$$ l_\gamma(\bm{w}-\mathcal{\tilde{C}}\bm{w}, \bm{v}) = 0, \  \forall \bm{w} \in V, \bm{v} \in V_\mathrm{f}.$$
Hence, the fine-scale projection operator $\mathcal{\tilde{C}}$ leads to an orthogonal splitting to the scalar product $l_\gamma$: 

$$ V = V_\mathrm{ms}  \oplus V_\mathrm{f} .$$
Any function $\bm{u} \in V $ can be decomposed into $\bm{u}_\mathrm{ms}  \in V_\mathrm{ms} $ and $ \bm{u}_\mathrm{f} \in V_\mathrm{f} ,  \bm{u} = \bm{u}_\mathrm{ms}  + \bm{u}_\mathrm{f}$, with $l_\gamma(\bm{u}_\mathrm{ms} , \bm{u}_\mathrm{f} ) = 0$. Since $\dim V_\mathrm{ms} = \dim V_H$, the space $V_\mathrm{ms}$ can be regarded as a modified coarse space. Replacing $V$ with $V_\mathrm{ms}$, we can now propose a new multiscale finite element method, that is, to find $\bm{u}_\mathrm{ms} \in V_\mathrm{ms}$ that satisfies
\begin{equation}
	\label{eq:ms_galerkin}
	\begin{cases}
		a(\bm{u}_\mathrm{ms}^{n},\bm{v}_u)-b(\bm{v}_u,\theta_\mathrm{ms}^{n})=(f^{n},\bm{v}_u),
		\\c(D_{\tau}\bm{u}_\mathrm{ms}^{n},v_\theta)+d(\theta_{h}^{n},v_\theta)+b(D_{\tau}\theta_\mathrm{ms}^{n},v_\theta)=(g^{n},v_\theta),\quad \forall (\bm{v}_u,v_\theta) \in V_\mathrm{ms}.
	\end{cases}
\end{equation}
Finally, we introduce a basis of $V_\mathrm{ms}$. The image of the nodal basis function $\bm{\lambda}_m$ under the fine-scale projection $\mathcal{\tilde{C}}$ is denoted by $\bm{\phi}_m=\mathcal{\tilde{C}} \bm{\lambda}_m \in V_{\mathrm{f}}$, i.e., $\bm{\phi}_m$ satisfies the corrector problem

$$
l_\gamma\left(\bm{\lambda}_m-\bm{\phi}_m, \bm{\bm{v}}\right)=0 \quad \forall \quad \bm{\bm{v}} \in V_\mathrm{f}.
$$
We emphasize that the corrector problem is posed in the fine-scale space $V_\mathrm{f}$, that is, the test and trial functions satisfy the constraint that their interpolation with respect to the coarse mesh vanishes. Then, a basis of $V_\mathrm{ms}$ is given by the modified nodal basis

$$
\left\{\bm{\psi}_m =I_\mathrm{ms}\bm{\lambda}_m = \bm{\lambda}_m-\bm{\phi}_m \mid m \in \mathcal{N}\right\},
$$ 
where $I_\mathrm{ms} \coloneqq I-\mathcal{\tilde{C}}$.
In general, the corrections $\bm{\psi}_m$ of the nodal basis functions $\bm{\lambda}_m$ have global support, a fact that limits the practical use of the modified basis and the corresponding method.
\subsection{Localization}
\label{subsec:localization}
By localization of the operator $I_\mathrm{ms}$, it is possible to deduce a practically feasible variant of the method Eq.(\ref{eq:ms_galerkin}). In \cite{LOD}, it is proved that the corrections decay exponentially, and a localization procedure is proposed. 

Before introducing ME-LOD, we start with a localization process, analogous to the technique adopted in most multiscale methods. Define patches of size $k$ in the following way, for $K \in \mathcal{T}_H$

$$
\begin{aligned}
	& \omega_0(K) \coloneqq \operatorname{int} K, \\
	& \omega_k(K) \coloneqq \operatorname{int}\left(\cup\left\{\hat{K} \in \mathcal{T}_H\mid \hat{K} \cap \overline{\omega_{k-1}(K)} \neq \emptyset\right\}\right), \quad k=1,2, \ldots
\end{aligned}
$$
and let $V_{\mathrm{f}}\left(\omega_k(K)\right) \coloneqq \left\{\bm{v} \in V_{\mathrm{f}}\mid v|_{\bar{\Omega} \backslash \omega_k(K)}=0\right\}$ be the restriction of $V_{\mathrm{f}}$ to the patch $\omega_k(K)$. Fig. \ref{fig:mesh} depicts the fine grid $\mathcal{T}_H$, the coarse grid $\mathcal{T}_H$ , the coarse element $K$ and neighborhood $\omega_2(K)$ of K. The solutions $\bm{\phi}_m^{k} \in V_\mathrm{f} (\omega_k(K))$ of  

$$
l_\gamma^{m,k}\left(\bm{\lambda}_m-\bm{\phi}_m^{k}, \bm{v}\right)=0, \quad \forall \bm{v} \in V_{\mathrm{f}}\left(\omega_k({K})\right), \quad {K} \in \mathcal{T}_H.
$$ 
are approximations of $\bm{\phi}_m$ with local support, where $l_\gamma^{m,k}$ is the the restrict of $l_\gamma$ to region $\omega_k(K)$. 
\begin{figure}
	\centering
	\includegraphics[scale=0.5]{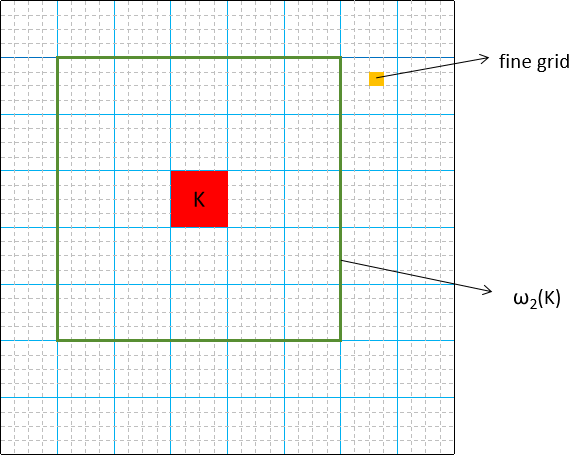}
	\caption{The fine grid $\Gamma_h$, the coarse grid $\Gamma_H$ , the coarse element K and neighborhood $\omega_2(K)$ of K }
	\label{fig:mesh}
\end{figure}
We can now define the localized multiscale space 

$$
V_\mathrm{ms}^{m,k}= \operatorname{span} \left\{\bm{\lambda}_m-\bm{\phi}_m^{k}\mid m \in \mathcal{N}\right\}, 
$$ 
and the localized operator $I_\mathrm{ms}^{m,k}:V_H\mapsto V_\mathrm{ms}^{m,k}.$

Since the dimension of $V_{\mathrm{f}}\left(\omega_k(K)\right)$ can be made significantly smaller than the dimension of $V_{\mathrm{f}}$ (depending on $k$), the problem of finding $I_{\mathrm{ms}}^{m,k} \bm{\lambda}_m$ is computationally cheaper than finding $I_{\mathrm{ms}}\bm{\lambda}_m$. Moreover, the resulting discrete system is sparse. It should also be noted that the computation of $I_{\mathrm{ms}}^{m,k} \bm{\lambda}_m$ for all nodes is suitable for parallelization since they are independent.
%
\subsection{Implementation of ME-LOD}
\label{subsec:implementation}
In this section, we focus on the numerical implementation of ME-LOD. As mentioned earlier, the construction of a local basis of the multiscale spaces is crucial. Our construction is based on a discrete subspace $V_H(\omega_k({K})) \subset V_h(\omega_k({K}))$ of dimension $N$, which has a local basis and satisfies 

$$I_\mathrm{ms}^{m,k}V_H(\omega_k(K)) = V_\mathrm{ms}^{m,k}.$$

The basis of $V_\mathrm{ms}^{m,k}$ can be obtained by two sets of saddle point problems. For $m \in J$, find $\bm{\psi}_m \in V(\omega_k(K))$ and $\tilde{\lambda} \in \mathbb{C}^N$ such that

$$
\begin{aligned}
	l_\gamma\left(\bm{\psi}_m, \bm{v}\right)+\sum_{j \in J} \tilde{\lambda}_j q_j(\bm{v})=0,\  & \forall \bm{v} \in V(\omega_k(K)), \\
	q_j\left(\bm{\psi}_m\right)=\delta_{j m},\  & \forall j \in J.
\end{aligned}
$$
The proof of this conclusion can be found in \cite{num_homo}.

Hence, the ME-LOD multiscale basis $\Psi_m$ can be calculated by the following equations

$$
\left[\begin{array}{cc}
	A&M\\
	M^T&O\\
\end{array}\right]
\left[\begin{array}{c}
	\Psi_m\\
	\Lambda\\
\end{array}\right]
=\left[\begin{array}{c}
	O\\
	I_j\\
\end{array}\right],
$$
where $A_j,M_j$ are the stiffness matrix and mass matrix defined on $\omega_k(K)$, $\Lambda = [ \tilde{\lambda}_1, \tilde{\lambda}_2,\cdots,  \tilde{\lambda}_J(m) ]^T$ and $I_j$ is a sparse vector taking 1 in the $j_{th}$ row. Like all classic multiscale methods, one can use the representation of multiscale basis functions via fine-scale basis functions to assemble the stiffness matrix. This is particularly useful in code development. Assume that multiscale basis function (in discrete form) $\psi_i$ can be written as 

$$\psi_i=r_{ij}\phi_j,$$
where $R = (r_{ij} )$ is a matrix and $\phi_j\in V_h$ are fine-scale finite element basis functions. The $i_{th}$ row of this matrix contains the fine-scale representation of the $i_{th}$ multiscale basis function. Here, it should be noted that the matrix $R$ only needs to be constructed once, and it can be repeatedly used for computation.
Define

$$
A_c^n=R^T A^n R, B^c=R^T BR, \mathbf{w}_c^n=\left(R^T R\right)^{-1} R^T \mathbf{w}^n, F_c^n= R^T F^n,
$$
and combining with Eq.(\ref{eq:iterationmatrix}), the following algebraic system can be given,

$$
A_c^n \mathbf{w}_c^n=B^c \mathbf{w}_c^{n-1}+F_c^n.
$$

Then the solution $\mathbf{w}_c^n$ can be calculated by iteration, and the solutions in fine grid can be obtained by use of the coupling multiscale basis functions,

$$
\mathbf{w}^n=R \mathbf{w}_c^n.
$$
Here, the key difference between ME-LOD and LOD is the construction of more accurate coupling multiscale basis functions. The matrix $R$ can be computed offline, and it can be reused in all time steps. Thus, a much smaller system is solved with ME-LOD, which can significantly improve computational efficiency.
\section{Convergence analysis}
\label{sec:error}
\subsection{Interpolation error}
Peterseim et al. present a unified variational framework for the LOD method of some general classes of linear PDEs in divergence form in \cite{num_homo}. The framework merges different concepts, which were successfully applied to multiscale problems. Therefore, this section wants to illustrate that the entire process of constructing multi-scale basis functions through auxiliary problems satisfies the abstract framework proposed in the article.

Notice that the $b(\cdot, \cdot)$ is continuous in $V_u(\Omega) \times L^2(\Omega)$, and it can be assumed that there exists a constant $C_0$ such that for all $\bm{u} \in V_u(\Omega), \theta \in L_2(\Omega)$, we have

$$
b^K(\bm{u}, \theta) \leq C_0\norm{\bm{u}}_{L_a, K}\norm{\theta}_{L_d, K}, \quad \forall K \in \mathcal{T}_H,$$
where

$$
\norm{\theta}_{L_d, K}^2=\int_K \kappa \theta^2.
$$
Also, define

$$
\norm{\bm{u}}_{L_a, K}^2=\int_K(\lambda+2 \mu) \bm{u} \cdot \bm{u}.
$$

For all $\bm{u} \in V_u(\Omega), \theta \in V_\theta(\Omega)$, the energy norms can be defined as follows

$$
\norm{\bm{u}}_{a, K}^2=a^K(\bm{u}, \bm{u}), \quad\norm{\theta}_{d, K}^2=d^K(\theta, \theta).
$$

The localized corrector $\tilde{\bm{w}} = I^{m,k}\bm{w} \in V_{\mathrm{ms}}^{m,k}$ can be defined as the solution to the problem:
\[
l_\gamma^{m,k}(\bm{u},\bm{v}) = l_\gamma^{m,k}(\bm{w},\bm{v}), \quad \forall \bm{v} \in V_{\mathrm{ms}}^{m,k}.
\]
To justify the validity and convergence of this construction, we refer to the general framework established in \cite{num_homo}, which proves numerical convergence and exponential decay of correctors under some assumptions: the domain is a bounded Lipschitz domain;  the test and trial spaces are closed subspaces of a Hilbert space;  the bilinear form is bounded and satisfies an inf-sup condition; the right-hand side is a continuous linear functional. These conditions are fully satisfied in our setting when assuming that there exists $\gamma_1,\gamma_2 \in (0,1]$ which makes $l_{\gamma}(\bm{u},\bm{v})$ satisfy. 
\begin{equation}
	\label{eq:inf_sup}
	0 < \alpha_{V}=\inf _{\substack{\bm{v} \in V \\ \bm{v} \neq 0}} \sup _{\substack{\bm{w} \in V \\ \bm{w} \neq 0}} \frac{|l_{\gamma}(\bm{u},\bm{v})|}{\norm{\bm{v}}_{V_u}\norm{\bm{w}}_{V_\theta}}=\inf _{\substack{\bm{w} \in V \\ \bm{w} \neq 0}} \sup _{\substack{\bm{v} \in V \\ \bm{v} \neq 0}} \frac{|l_{\gamma}(\bm{u},\bm{v})|}{\norm{\bm{v}}_{V_u}\norm{\bm{w}}_{V_\theta}}.
\end{equation}

Hence, problem $\ref{variation}$ is well-posed, admits the stability estimate and fits into the abstract framework of \cite{num_homo}.

Therefore, the convergence and localization results in \cite{GMsFEM} apply directly to our multiscale space construction.

\begin{equation}
	\norm{\bm{w}-\tilde{\bm{w}}}_{l_\gamma^{m,k}} \lesssim C(H,k)\norm{L_\gamma^{m,k}(\bm{w}) }_{L^2(\omega_k(K))},
\end{equation}
where $C(H,k)=\left( H+H^{-2} \exp(-ck) \right)$. 

The quality of function approximation in the multiscale space improves with an increasing number of oversampling layers $k$, but the benefit saturates when it becomes sufficiently large. This holds for both the energy and $L^2$-norms.
\begin{lemma} For all $ \bm{w} \in V(\omega_k(K)), \tilde{\bm{w}} = I^{m,k} \bm{w}\in V_\mathrm{ms}^{m,k},$
	\begin{equation}
		\label{eq:lem1}
		\norm{\bm{w}-\tilde{\bm{w}}}_{L^2(\omega_k(K))} \lesssim C^2(H,k)\norm{L_\gamma^{m,k}(\bm{w}) }_{L^2(\omega_k(K))} ,
	\end{equation}
	where $C(H,k)=H+H^{-2}\exp(-ck)$.
\end{lemma}
\begin{proof}
	
	Consider a differential operator $L_\gamma^*$ with 
	
	$$
	l_\gamma^{m,k}(\bm{u},\bm{v})= \left(L_\gamma(\bm{u}),\bm{v}\right)=\left(\bm{u},L_\gamma^*(\bm{v})\right).
	$$
	The problem
	\begin{equation}\label{eq:dual}\begin{cases}L_\gamma^*(\bm{u})=\bm{w}-\tilde{\bm{w}},&\mathrm{in }\quad\Omega,\\\bm{u}=0,&\mathrm{on }\quad\partial\Omega,\end{cases}		\end{equation}
	has the week formulation $l_\gamma^{m,k}(\bm{v},\bm{u})=(\bm{w}-\tilde{\bm{w}},\bm{v}),\  \forall \bm{v} \in V_\mathrm{ms}$. Then, using Nitsche?s trick:
	
	$$
	\begin{aligned}
		\norm{\bm{w}-\tilde{\bm{w}}}_{L^2(D)}^2=&\ (\bm{w}-\tilde{\bm{w}},\bm{w}-\tilde{\bm{w}})=\ l_\gamma^{m,k}(\bm{w}-\tilde{\bm{w}},\bm{u})\\
		=&\ l_\gamma^{m,k}(\bm{w}-\tilde{\bm{w}},\bm{u}-I_\mathrm{ms} \bm{u})+l_\gamma^{m,k}(\bm{w}-\tilde{\bm{w}},I_\mathrm{ms} \bm{u} )\\
		\leq &\ \norm{\bm{w}-\tilde{\bm{w}}}_{l_\gamma^{m,k}}\norm{\bm{u}-I_\mathrm{ms} \bm{u}}_{l_\gamma^{m,k}}\\
		\lesssim &\ C^2(H,k)\norm{L_\gamma^{m,k}(\bm{w})}_{L^2(\omega_k(K))}\norm{\bm{w}-\tilde{\bm{w}}}_{L^2(\Omega)}.
	\end{aligned}
	$$
\end{proof}
Hence, the error estimation of the multiscale interpolation operator is given by the following theorem.
\begin{theorem}For all $\bm{w} \in  V_u (\Omega) \times V_\theta (\Omega)$ and definition of $\mathcal{C}_{\theta H}$, the local interpolation error can be given as
	\begin{equation}
		\label{eq:thm1}  
		\begin{aligned}
			&\quad \norm{ \bm{u}-  \left(I_\mathrm{ms} \bm{w}\right)_u\norm{_{a,\Omega}^2+} \theta-\left(I_\mathrm{ms} \bm{w}\right)_\theta }_{d, \Omega}^2  \\
			& \lesssim \left(C_2  C(H,k)+C_1 C^2(H,K)\right)\sum_{\omega_k(K)}\norm{L_\gamma^{m,k}(\bm{w})}_{L^2(\omega_k(K))}, 
		\end{aligned}
	\end{equation}
\end{theorem}
where we define $(I_\mathrm{ms}\bm{w})_u$ and $(I_\mathrm{ms}\bm{w})_\theta$  s.t. $I_\mathrm{ms}\bm{w} = ((I_\mathrm{ms}\bm{w})_u , (I_\mathrm{ms}\bm{w})_\theta)$.
\begin{proof}
	From lemma 3.3 of \cite{regularize}, it can be concluded that
	
	$$
	\begin{aligned}
		\label{eq:Kerror}
		\norm{\bm{u}-\left(I_\mathrm{ms} \bm{w}\right)_u}_{a, K}+\norm{\theta-\left(I_\mathrm{ms} \bm{w}\right)_\theta}_{d, K} \lesssim \sum_{y_i \in K}\left(C_2 \norm{\bm{w}-I_\mathrm{ms}\bm{w}}_{l_\gamma(K)} \right. \\ \left.+C_1\left(\norm{\bm{u}-\left(I_\mathrm{ms} \bm{w}\right)_u}_{L_a, K}+\norm{\theta-\left(I_\mathrm{ms} \bm{w}\right)_\theta}_{L_d,K}\right)\right),
	\end{aligned}
	$$
	where $C_1$ and $C_2$ are constants that related to $\gamma,C_0$ but independent of $H$. The following estimations can be given as
	\begin{equation}
		\label{eq:Eerror_interpolatation}
		\begin{aligned}
			& \sum_{K \in \mathcal{T}^H} \sum_{y_i \in K} \norm{\bm{w}-I_\mathrm{ms}\bm{w}}_{l_\gamma(K)}=\sum_{\omega_k(K)} \norm{\bm{w}-I_\mathrm{ms}^{m,k} \bm{w}}_{l_\gamma(K)} \\
			& \lesssim \sum_{\omega_k(K)} C(H,k)\norm{ \bm{w} }_{2, \omega_k(K)}^2, 
		\end{aligned}
	\end{equation}
	\begin{equation}
		\label{eq:L2error_interpolatation}
		\begin{aligned}
			& \sum_{K \in \mathcal{T}^H} \sum_{y_i \in K}\left(\norm{\bm{u}-\left(I_\mathrm{ms}^{m,k} \bm{w}\right)_u}_{L_a, K}+\norm{\theta-\left(I_\mathrm{ms}^{m,k} \bm{w}\right)_\theta}_{L_d, K}\right) \\
			&= \sum_{\omega_k(K)}\left(\norm{\bm{u}-\left(I_\mathrm{ms}^{m,k} \bm{w}\right)_u}_{L_a, \omega_k(K)}+\norm{\theta-\left(I_\mathrm{ms}^{m,k} \bm{w}\right)_\theta}_{L_d, \omega_k(K)}\right) \\
			&\lesssim\sum_{\omega_k(K)}\norm{\bm{w}-I_\mathrm{ms}^{m,k} \bm{w}}_{L^2(\omega_k(K))}^2 \lesssim \sum_{\omega_k(K)}  C^2(H,k)\norm{L_\gamma^{m,k}(\bm{w}) }_{L^2(\omega_k(K))} .
		\end{aligned}
	\end{equation}
	Then, combining the summation of  Eq.(\ref{eq:Kerror}) for all $K \in \Gamma_H$ with  Eq.(\ref{eq:Eerror_interpolatation}) and  Eq.(\ref{eq:L2error_interpolatation}), the proof is complete.
\end{proof}

\subsection{Steady state case}
We begin by presenting the a priori error estimates for the steady-state formulation of the coupled thermoelastic problem. Let $\overline{\bm{u}} \in V_u(\Omega)$ and $\bar{\theta} \in V_\theta(\Omega)$ denote the exact solutions to the continuous problem. The variational formulation reads as follows:
\begin{equation}
	\begin{aligned}
		a\left(\overline{\bm{u}}, \bm{v}_u\right) - b\left(\bm{v}_u, \bar{\theta}\right) &= \left\langle\overline{\bm{f}}, \bm{v}_u\right\rangle_a, && \forall \bm{v}_u \in V_u(\Omega), \\
		d\left(\bar{\theta}, v_\theta\right) &= \left\langle\bar{g}, v_\theta\right\rangle_d, && \forall v_\theta \in V_\theta(\Omega).
	\end{aligned}
	\label{eq:aux}
\end{equation}

The discrete version of system~\eqref{eq:aux} is then written as:
\begin{equation}
	\begin{aligned}
		a\left(\overline{\bm{u}}_H, \bm{v}_{u H}\right) - b\left(\bm{v}_{u H}, \bar{\theta}_H\right) &= \left\langle\overline{\bm{f}}, \bm{v}_{u H}\right\rangle_a, && \forall \bm{v}_{u H} \in V_{u,H}, \\
		d\left(\bar{\theta}_H, v_{\theta H}\right) &= \left\langle\bar{g}, v_{\theta H}\right\rangle_d, && \forall v_{\theta H} \in V_{\theta,H}.
	\end{aligned}
\end{equation}

To analyze the discretization error, we introduce a Riesz projection operator $\mathscr{R}_H$ that maps the continuous solution $(\bm{u}, \theta) \in V_u(\Omega) \times V_\theta(\Omega)$ onto the discrete space $V_{u,H} \times V_{\theta,H}$:
\begin{equation}
	\begin{aligned}
		a\left(\bm{u} - \mathscr{R}_{u H}(\bm{u}, \theta), \bm{v}_{u H}\right) - b\left(\bm{v}_{u H}, \theta - \mathscr{R}_{\theta H}(\theta)\right) &= 0, && \forall \bm{v}_{u H} \in V_{u,H}, \\
		d\left(\theta - \mathscr{R}_{\theta H}(\theta), v_{\theta H}\right) &= 0, && \forall v_{\theta H} \in V_{\theta,H}.
	\end{aligned}
\end{equation}

\begin{lemma}
	Let $(\bm{u}, \theta) \in V_u(\Omega) \times V_\theta(\Omega)$. Then the Riesz projection $\mathscr{R}_H$ satisfies the error bound:
	\begin{equation}
		\begin{aligned}
			\norm{\bm{u} - \mathscr{R}_{u H}(\bm{u}, \theta)}_a &\leq \inf_{\bm{v}_{u H} \in V_{u,H}} \norm{\bm{u} - \bm{v}_{u H}}_a + C_0 \norm{\theta - \mathscr{R}_{\theta H}(\theta)}_c, \\
			\norm{\theta - \mathscr{R}_{\theta H}(\theta)}_d &\leq \inf_{v_{\theta H} \in V_{\theta,H}} \norm{\theta - v_{\theta H}}_d.
		\end{aligned}
		\label{eq:inflem1}
	\end{equation}
\end{lemma}

\begin{lemma}
	\label{lem:dual}
	Let $r \in L_d$, and define $\phi \in V_\theta(\Omega)$ as the solution of the dual problem:
	\[
	d(\phi, v_\theta) = c(r, v_\theta), \quad \forall v_\theta \in V_\theta(\Omega),
	\]
	and let $\phi_H \in V_{\theta,H}$ be its discrete approximation. Then,
	\[
	\norm{\phi - \phi_H}_d \leq C_3 \norm{r}_c, \quad \text{where } C_3 = C_p^{1/2} \kappa_{\min}^{-1/2}.
	\]
\end{lemma}
\begin{proof}
	The estimate follows by testing the dual problem with $\phi - \phi_H$ and applying the Poincar\'{e} inequality:
	\begin{equation}
		\begin{aligned}
			\norm{\phi - \phi_H}_d^2 &= d(\phi, \phi - \phi_H) = c(r, \phi - \phi_H) \leq \norm{r}_c \norm{\phi - \phi_H}_c, \\
			\norm{\phi - \phi_H}_d^2 &\geq \kappa_{\min} \norm{\nabla(\phi - \phi_H)}^2 \geq C_p^{-1} \kappa_{\min} \norm{\phi - \phi_H}_c^2.
		\end{aligned}
		\label{eq:duals}
	\end{equation}
	Combining the inequalities yields the desired result.
\end{proof}

\begin{lemma}
	\label{lem:projection}
	Let $\bm{w} = (\bm{u}, \theta) \in V_u(\Omega) \times V_\theta(\Omega)$. Then the following a priori error estimates hold:
	\begin{equation}
		\begin{aligned}
			\norm{\bm{u} - \mathscr{R}_{u H}(\bm{u}, \theta)}_a &\lesssim \max\left\{1, C_0 C_3\right\} \tilde{C}(H,k) \norm{L_\gamma^{m,k}(\bm{w})}_{L^2(\Omega)}, \\
			\norm{\theta - \mathscr{R}_{\theta H}(\theta)}_d &\lesssim \tilde{C}(H,k) \norm{L_\gamma^{m,k}(\bm{w})}_{L^2(\Omega)}, \\
			\norm{\theta - \mathscr{R}_{\theta H}(\theta)}_c &\lesssim C_3 \tilde{C}(H,k) \norm{L_\gamma^{m,k}(\bm{w})}_{L^2(\Omega)},
		\end{aligned}
		\label{eq:projection}
	\end{equation}
	where
	\begin{align*}
		&\tilde{C}(H,k) := \left[C_2 C(H,k) + C_1 C^2(H,k)\right]^{1/2}, \\
		&\norm{L_\gamma^{m,k}(\bm{w})}_{L^2(\Omega)} := \sum_{\omega_k(K)} \norm{L_\gamma^{m,k}(\bm{w})}_{L^2(\omega_k(K))}.
	\end{align*}
\end{lemma}
\begin{proof}
	Let $r := \theta - \mathscr{R}_{\theta H}(\theta)$ and consider
	\[
	\norm{\theta - \mathscr{R}_{\theta H}(\theta)}_c^2 = c(r, r) = d(\phi, r) = d(\phi - \phi_H, r) \leq \norm{\phi - \phi_H}_d \norm{r}_d.
	\]
	Applying Lemma~\ref{lem:dual}, we obtain
	\begin{equation}
		\norm{\theta - \mathscr{R}_{\theta H}(\theta)}_c \leq C_1 \norm{\theta - \mathscr{R}_{\theta H}(\theta)}_d.
		\label{eq:normcd}
	\end{equation}
	Substituting this into~\eqref{eq:inflem1} completes the proof.
\end{proof}

\subsection{The Prior Error Estimate of ME-LOD}
For $\bm{f}$ and $g$ in problem \ref{eq}, we have

$$
\left\langle \tilde{\bm{f}}, \bm{v}_u \right\rangle_a = \left\langle \bm{f}, \bm{v}_u \right\rangle, \quad \forall \bm{v}_u \in V_u(\Omega), \quad \left\langle \tilde{g}, v_\theta \right\rangle_d = \left\langle g, v_\theta \right\rangle, \quad \forall v_\theta \in V_\theta(\Omega),
$$
where $\tilde{\bm{f}} \in V_u(\Omega)$ and $\tilde{g} \in V_\theta(\Omega)$. Define

$$
C^n(\bm{f}, g) = \frac{1}{2} \norm{ \tilde{\bm{f}}^n - \tilde{\bm{f}}_H^n }_a^2 + \tau_n \norm{ \tilde{g}^n - \tilde{g}_H^n }_d^2, \quad \forall n \in \{1, \dots, N_T\}.
$$
For simplicity of notation, let

$$
\begin{aligned}
	& C_1^n(\bm{w}) = 4 C_3^2 \left(C_3^2 + C_0^2 \max \left\{ 1, C_3^2 C_0^2 \right\} \right) \norm{ \partial_t \bm{w} }_{L^\infty(T_n, \norm{\cdot}_{2, \Omega})}, \\
	& C_2^n(\bm{w}) = 2 C_3^2 \left( \norm{ \partial_{tt} \theta }_{L^\infty(T_n, \norm{\cdot}_c)} + C_0^2 \norm{ \partial_{tt} \bm{u} }_{L^\infty(T_n, \norm{\cdot}_a)} \right).
\end{aligned}
$$
Then we have the following prior error estimate of the ME-LOD.

\begin{theorem}
	Let $\bm{w} = (\bm{u}, \theta)$ and $\bm{w}_H = (\bm{u}_H, \theta_H)$ be the unique solution and ME-LOD solution of the problem (equation), then there holds
	
	$$
	\begin{aligned}
		& \quad  \norm{ \bm{u}^n - \bm{u}_H^n }_a^2 + \norm{ \theta^n - \theta_H^n }_c^2 \\
		& \lesssim \sum_{m=1}^n \big[C^m(\bm{f}, g) + \tau_m \tilde{C}^2(H, k) C_1^m(\bm{w})  \\
		& \quad + \tau_m^3 C_2^m(\bm{w}) \big] 
		+ \frac{1}{2} \left( C_3^2 + \max \left\{ 1, C_3^2 C_0^2 \right\} \right) \tilde{C}^2(H, k) \norm{ l_\gamma^{m, k} (\bm{w}) }_{L^2(\Omega)}^2,
	\end{aligned}
	$$
	and
	
	$$
	\begin{aligned}
		\sum_{m=1}^n  \tau_m \norm{ \theta^m - \theta_H^m }_d^2 
		& \lesssim \sum_{m=1}^n \left[ C^m(\bm{f}, g) + \tau_m \tilde{C}^2(H, k) C_1^m(\bm{w}) \right. \\
		& \quad \left. + \tau_m^3 C_2^m(\bm{w}) \right]
		+ \sum_{m=1}^n \frac{1}{4} \tau_m \tilde{C}^2(H, k) \norm{ l_\gamma^{m, k} (\bm{w}) }_{L^2(\Omega)}^2.
	\end{aligned}
	$$
	where $n \in \{1, 2, \dots, N_T\}$.
\end{theorem}

\begin{proof}
	By the definition of the Riesz projection operator $\mathscr{R}_H$, we define
	
	$$
	\eta_{u H}^n = \mathscr{R}_{u H}(\bm{u}^n, \theta^n) - \bm{u}_H^n, \quad \eta_{\theta H}^n = \mathscr{R}_{\theta H}(\theta^n) - \theta_H^n.
	$$
	Combining Eq.(\ref{eq:discrete}), we have
	
	$$
	\begin{aligned}
		&\quad a(\eta_{u H}^n, \bm{v}_{u H}) - b(\bm{v}_{u H}, \eta_{\theta H}^n) \\
		&\qquad + c(\eta_{\theta H}^n - \eta_{\theta H}^{n-1}, v_{\theta H}) + b(\eta_{u H}^n - \eta_{u H}^{n-1}, v_{\theta H})  + \tau_n d(\eta_{\theta H}^n, v_{\theta H}) \\
		& = \left\langle \tilde{\bm{f}}^n - \tilde{\bm{f}}_H^n, \bm{v}_{u H} \right\rangle_a + \tau_n \left\langle \tilde{g}^n - \tilde{g}_H^n, v_{\theta H} \right\rangle_d \\
		& \qquad + c(\delta_{\theta H}^n, v_{\theta H}) + b(\delta_{u H}^n, v_{\theta H}), \quad \forall (\bm{v}_{u H}, v_{\theta H}) \in V_{\text{cgm}},
	\end{aligned}
	$$
	where
	
	$$
	\begin{aligned}
		& \delta_{\theta H}^n = \mathscr{R}_{\theta H}(\theta^n) - \mathscr{R}_{\theta H}(\theta^{n-1}) - \tau_n \partial_t \theta^n, \\
		& \delta_{u H}^n = \mathscr{R}_{u H}(\bm{u}^n, \theta^n) - \mathscr{R}_{u H}(\bm{u}^{n-1}, \theta^{n-1}) - \tau_n \partial_t \bm{u}^n.
	\end{aligned}
	$$
	
	Define $v_{u H} = \eta_{u H}^n - \eta_{u H}^{n-1} \in V_{u,H}$ and $v_{\theta H} = \eta_{\theta H}^n$, it follows
	\begin{equation}
		\begin{aligned}
			& \frac{1}{2} \norm{\eta_{u H}^n}_a^2 + \frac{1}{2} \norm{\eta_{u H}^n - \eta_{u H}^{n-1}}_a^2 + \frac{1}{2} \norm{\eta_{\theta H}^n}_c^2 + \frac{1}{2} \norm{\eta_{\theta H}^n - \eta_{\theta H}^{n-1}}_c^2 + \tau_n \norm{\eta_{H d \theta}^n}_d^2 \\
			& = \frac{1}{2} \norm{\eta_{u H}^{n-1}}_a^2 + \frac{1}{2} \norm{\eta_{\theta H}^{n-1}}_c^2 + \left\langle \tilde{\bm{f}}^n - \tilde{\bm{f}}_H^n, \eta_{u H}^n - \eta_{u H}^{n-1} \right\rangle_a \\
			& + \tau_n \left\langle \tilde{g}^n - \tilde{g}_H^n, \eta_{\theta H}^n \right\rangle_d + c(\delta_{\theta H}^n, \eta_{\theta H}^n) + b(\delta_{u H}^n, \eta_{u H}^n).
		\end{aligned}
		\label{eq:reserror}
	\end{equation}
	
	Similar to Eq.(\ref{eq:duals}), by using $\norm{\eta_{\theta H}^n}_c \leq C_1 \norm{\eta_{\theta H}^n}_d$, we have
	
	$$
	\begin{aligned}
		\frac{1}{2} \norm{\eta_{u H}^n}_a^2 + \frac{1}{2} \norm{\eta_{\theta H}^n}_c^2 + \frac{1}{4} \tau_n \norm{\eta_{\theta H}^n}_d^2 
		& \leq \frac{1}{2} \norm{\eta_{u H}^{n-1}}_a^2 + \frac{1}{2} \norm{\eta_{\theta H}^{n-1}}_c^2 + C^n(\bm{f}, g) \\
		& + 4 C_3^2 \tau_n^{-1} \norm{\delta_{\theta H}^n}_c^2 + 4 C_3^2 \tau_n^{-1} C_0^2 \norm{\delta_{u H}^n}_a^2.
	\end{aligned}
	$$
	Based on the regularity assumption of the problem and Eq.(\ref{eq:projection}), we have
	\begin{equation}
		\begin{aligned}
			\norm{\delta_{\theta H}^n}_c & = \norm{ - \int_{T_n} \left[ \partial_t \theta(s) - \mathscr{R}_{\theta H} (\partial_t \theta(s)) \right] ds - \int_{T_n} \left( s - t_{n-1} \right) \partial_{tt} \theta(s) dt }_c \\
			& \leq \tau_n C_3 \tilde{C}(H,k) \norm{\partial_t \bm{w}}_{L^\infty(T_n, \norm{\cdot}_{2, \Omega})} + \frac{1}{2} \tau_n^2 \norm{\partial_{tt} \theta}_{L^\infty(T_n, \norm{\cdot}_c)}.
		\end{aligned}
		\label{eq:delta_theta}
	\end{equation}
	
	Similarly, using Eq.(\ref{eq:projection}), we obtain
	\begin{equation}
		\norm{\delta_{u H}^n}_a \leq \tau_n \max \left\{ 1, C_3 C_0 \right\} \tilde{C}(H,k) \norm{\partial_t \bm{w}}_{L^\infty(T_n, \norm{\cdot}_{2, \Omega})}.
	\end{equation}
	Thus, by substituting these estimates into Eq.(\ref{eq:reserror}), we obtain the error estimate.
\end{proof}
This section proves that a coupled LOD algorithm with a regularization coefficient $\gamma$ is convergent, which is to satisfy the inf-sup condition for static problems. It is not necessary to introduce a special $\gamma$ in the algorithm, that is  $\gamma_1=\gamma_2=1$, to achieve good results. The appropriate selection of $\gamma$ will also lead to some slight optimization of the algorithm's accuracy. However, this conclusion drawn from numerical experiments has not yet been proven.

\section{Numerical experiments}
\label{sec:experiments}
In this section, we perform three numerical experiments to evaluate the performance of the proposed ME-LOD methods. The computation domain is set to be the unit square $\Omega = [0,1]\times [0,1]$ and both the temperature and the displacement have a homogeneous dirichlet boundary condition, that is $\mathcal{T}_\theta^D = \mathcal{T}_{\bm{u}}^D =\partial \Omega$ and $\mathcal{T}_\theta^N = \mathcal{T}_{\bm{u}}^N = \emptyset $.

The solutions $(\bm{u}^{melod},\theta^{melod})$ of ME-LOD will be compared with the reference solutions $(\bm{u}^{ref},\theta^{ref})$ of the standard finite element method and the solutions $(\bm{u}^{lod}, \theta^{lod})$ of LOD. Then, the relative energy errors of each solution, the total relative energy errors are defined as follows

\begin{equation}
	\label{eq:errors}
	\begin{aligned}
		\norm{E_u}_e= & \frac{\left(\int_{\Omega} \sigma\left(E_u\right): \epsilon\left(E_u\right) d x\right)^{\frac{1}{2}}}{\left(\int_{\Omega} \sigma(\bm{u}): \epsilon(\bm{u}) d x\right)^{\frac{1}{2}}},\norm{E_\theta}_e=\frac{\left(\int_{\Omega} \kappa \nabla E_\theta \cdot \nabla E_\theta d x\right)^{\frac{1}{2}}}{\left(\int_{\Omega} \kappa \nabla \theta \cdot \nabla \theta d x\right)^{\frac{1}{2}}}, \\
		\norm{E_w}_e= & \frac{\left(\int_{\Omega} \sigma\left(E_u\right): \epsilon\left(E_u\right)+\kappa \nabla E_\theta \cdot \nabla E_\theta d x\right)^{\frac{1}{2}}}{\left(\int_{\Omega} \sigma(\bm{u}): \epsilon(\bm{u})+\kappa \nabla \theta \cdot \nabla \theta d x\right)^{\frac{1}{2}}},
	\end{aligned}
\end{equation}
where $E_\theta$ represents the temperature error  $E_{\theta}^{melod} = \theta^{melod}-\theta^{ref} $ or $ E_{\theta}^{lod} = \theta^{lod}-\theta^{ref}$ and $E_u$ represents the displacement error $E_{u}^{melod}=\bm{u}^{melod}-\bm{u}^{ref} $ or $E_{u}^{lod}=\bm{u}^{lod}-\bm{u}^{ref}$. 


\subsection{Convergence of ME-LOD}
\label{subsec:test1}

In this test, heterogeneous media have random microstructure and coefficients. The body force $\bm{f}$ and heat source $g$ are chosen as

$$
\bm{f} = (0,0) ,\quad g = 10.
$$
The initial boundary condition is defined as

$$
\theta_0 = 500\sin(\pi x)\sin(\pi y).
$$ The material coefficients $\kappa(\bm{x}; \xi)$, $ \lambda(\bm{x}; \xi)$, $\mu(\bm{x}; \xi)$, and $\beta(\bm{x}; \xi)$ satisfy the following logarithmic Gaussian random field 

$$
\exp \left(\mathcal{G P}\left(b_0(\bm{x}), \operatorname{Cov}\left(\bm{x}_1, \bm{x}_2\right)\right)\right),
$$
where $\operatorname{Cov}(x_1,x_2) = \sigma^2 \exp(-\norm{x_1-x_2}^2 / l^2)$, $x_1$ and $x_2$ are spatial coordinates in $\Omega$, $\sigma^2$ is the overall variance, and $l$ is the length scale. Then the coefficients are shown on the left of Fig. \ref{fig:3coe}.

\begin{figure}[htbp]
	\centering
	\includegraphics[scale=0.5]{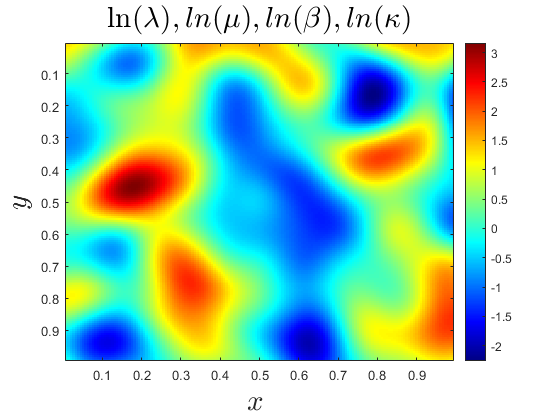}
	\caption{Contour plots of the material coefficients in the random microstructure. Lam\'{e} coefficients $\mu$, $\lambda$, thermal conductivity coefficient $\kappa$, and expansion coefficient $\beta$.}
	\label{fig:3coe}
\end{figure}
The reference solution is computed on a mesh of $h = \sqrt{2} \cdot 2^{-7}$ and the ME-LOD is computed for five decreasing values of the mesh size, namely, $H = \sqrt{2} \cdot 2^{-2}, \sqrt{2} \cdot 2^{-3}, ...,\sqrt{2} \cdot 2^{-5}$. The patch sizes $k$ are chosen such that $k \sim \log(H^{-1})$, that is $k = 1, 2, 3, 4$. Then, the relative errors are shown in Fig. \ref{fig:3error}, where the left graph shows the relative errors for the displacement and the right graph shows the error for the temperature. As expected, the ME-LOD shows a convergence of optimal order.

\begin{figure}[htpb]
	\centering
	
	\hfill
	\begin{subfigure}[b]{0.95\textwidth}
		\includegraphics[width=0.47\textwidth]{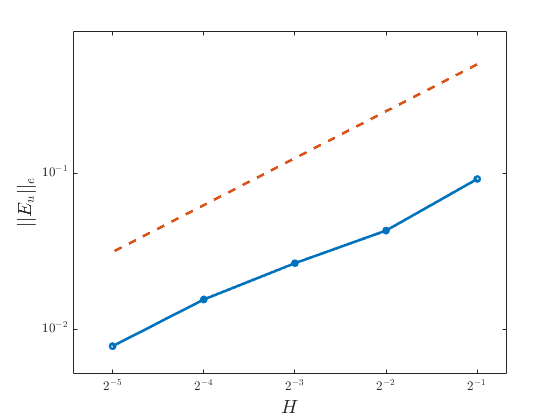}
		\hfill
		\includegraphics[width=0.47\textwidth]{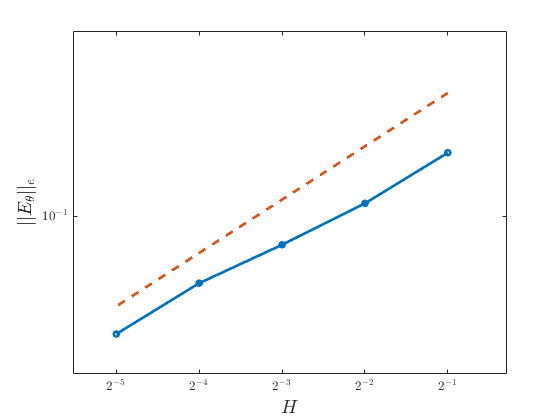}
	\end{subfigure}
	\hfill
	
	\caption{Relative errors using ME-LOD plotted against the mesh size $H$. The dashed line is $H$.}
	\label{fig:3error}
\end{figure}

\subsection{Comparison of LOD and ME-LOD on periodic structures}
In this example, the efficiency and superiority of ME-LOD will be verified for heterogeneous media with periodic microstructure. The initial conditions, boundary conditions, and source terms are the same as in subsection \ref{subsec:test1}. The Lam\'{e} coefficients $\mu,\lambda$, thermal conductivity coefficient $\kappa$, and expansion coefficient $\alpha$ are shown in Fig. \ref{fig:1coe}, where the contrast is chosen as $\lambda_{\max}:\lambda_{\min} = \mu_{\max}:\mu_{\min} = \kappa_{\max}:\kappa_{\min} =\alpha_{\max}:\alpha_{\min} = 10^3:1$. The reference solution is calculated on a fine mesh of $h = \sqrt{2} \times 2^{-7}$, and the multiscale basis is defined on a coarse mesh of $H = \sqrt{2} \times 2^{-4}$.

\begin{figure}[htbp]
	\centering
	\includegraphics[scale=0.5]{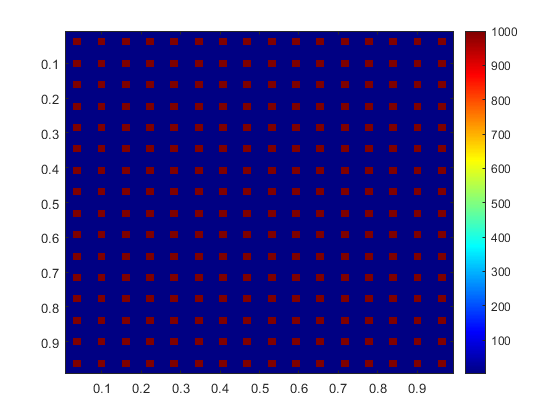}
	\caption{Contour plots of the material coefficients in periodic microstructure. Lam\'{e} coefficients $\mu$, $\lambda$, thermal conductivity coefficient $\kappa$, and expansion coefficient $\beta$.}
	\label{fig:1coe}
\end{figure}

\begin{figure}[htpb]
	
	\centering
	\begin{subfigure}[b]{0.95\textwidth}
		\includegraphics[width=0.32\textwidth]{1u1\_fem4}
		\hfill
		\includegraphics[width=0.32\textwidth]{1u2\_fem4}
		\hfill
		\includegraphics[width=0.32\textwidth]{1u3\_fem4}
		\caption{FEM}
	\end{subfigure}
	\hfill
	\begin{subfigure}[b]{0.95\textwidth}
		\includegraphics[width=0.32\textwidth]{1u1\_lod4}
		\hfill
		\includegraphics[width=0.32\textwidth]{1u2\_lod4}
		\hfill
		\includegraphics[width=0.32\textwidth]{1u3\_lod4}
		\caption{LOD}
	\end{subfigure}
	
	\begin{subfigure}[b]{0.95\textwidth}
		\includegraphics[width=0.32\textwidth]{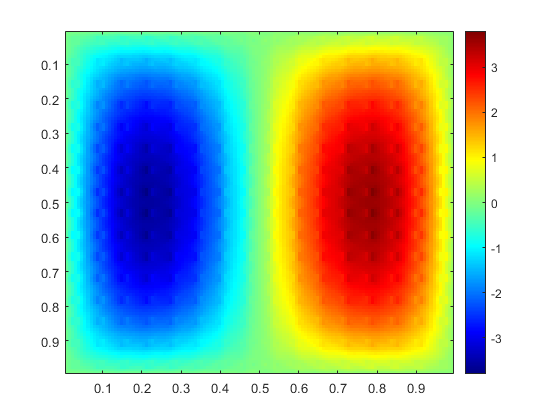}
		\hfill
		\includegraphics[width=0.32\textwidth]{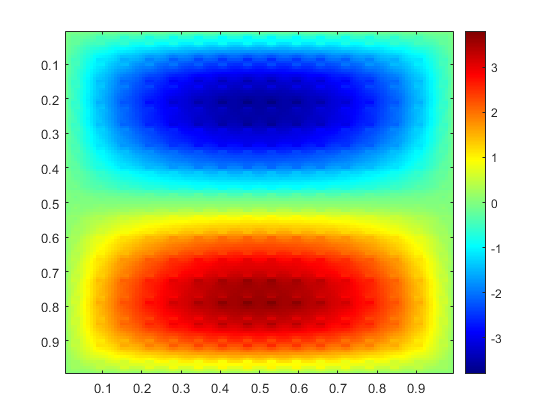}
		\hfill
		\includegraphics[width=0.32\textwidth]{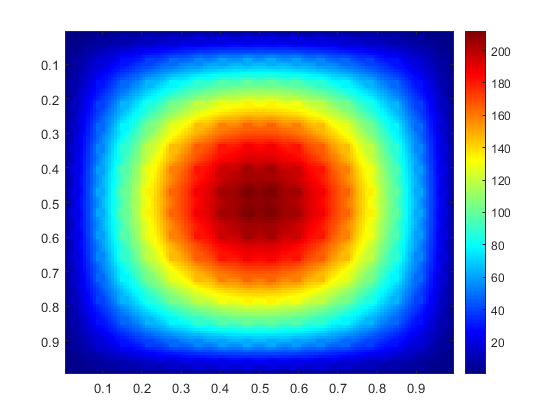}
		\caption{ME-LOD}
	\end{subfigure}
	\hfill
	\caption{Contour plots of the numerical solution of normal FEM, LOD method, and ME-LOD method}
	\label{1sol}
\end{figure}
Fig. \ref{1sol} demonstrates the contour plots of reference solutions $\left(\bm{u}^{r e f}, \theta^{r e f}\right)$, ME-LOD solutions $\left(\bm{u}^{melod}, \theta^{melod}\right)$, and LOD solutions $\left(\bm{u}^{lod}, \theta^{lod}\right)$. It can be concluded that the proposed ME-LOD has higher accuracy than the LOD with the same size of local multiscale basis functions. For comparison purposes, we calculate the energy errors of LOD and ME-LOD as defined by  Eq.(\ref{eq:errors}). The numerical results for different extension numbers $k=2,3,4,5$ of multiscale basis functions are given in Table 1. 

\begin{table}[htbp]
	\centering
	\caption{Relative energy errors with different patch sizes for periodic microstructure.}
	\begin{tabular}{cllll}
		\hline patch size & $\norm{E_w^{melod}}_e$ & $\norm{E_w^{lod}}_e$ & $\norm{E_w^{melod}}_{L^2}$ & $\norm{E_w^{lod}}_{L^2}$  \\
		\hline 
		
		$k=2$ & 2.18E-01 & 8.27E-01 & 9.36E-02 & 6.92E-01 \\
		$k=3$ & 5.98E-02 & 4.48E-01 & 7.31E-03 & 1.29E-01 \\
		$k=4$ & 1.67E-02 & 3.08E-01 & 3.07E-03 & 2.72E-02 \\
		$k=5$ & 6.65E-03 & 2.54E-01 & 4.20E-03 & 1.04E-02 \\
		\hline
	\end{tabular}
\end{table}

From the table, we observe that at the beginning, the total relative energy error $\norm{E_w^{melod}}_e$ for ME-LOD decreases with the increase of $k$, but as the number of oversampling layers $k$ gets larger, which means a larger computational load, the error decay slowly. This situation indicates the newly constructed multiscale basis function has better local properties, that is, they can obtain enough information within a small local area. Moreover,  it can also clearly be found that the energy errors and $L^2(\Omega)$ of ME-LOD with $k = 2$ are obviously smaller than those of LOD with $k = 5$ in Fig. \ref{fig:1error}, which demonstrate that the ME-LOD is more efficient than the LOD.

\begin{figure}[htbp]
	\centering
	\begin{subfigure}[b]{0.49\textwidth}
		\centering
		\includegraphics[width=\textwidth]{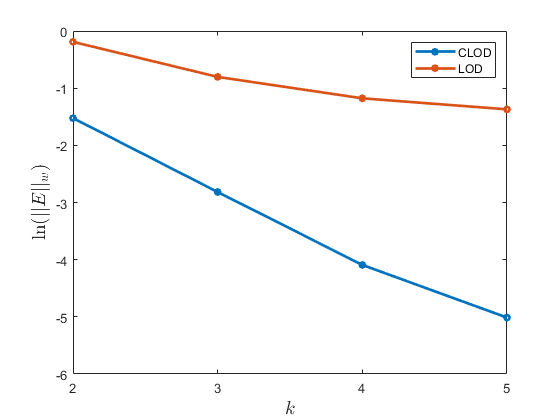}
		\caption{}
	\end{subfigure}
	\hfill 
	\begin{subfigure}[b]{0.49\textwidth}
		\centering
		\includegraphics[width=\textwidth]{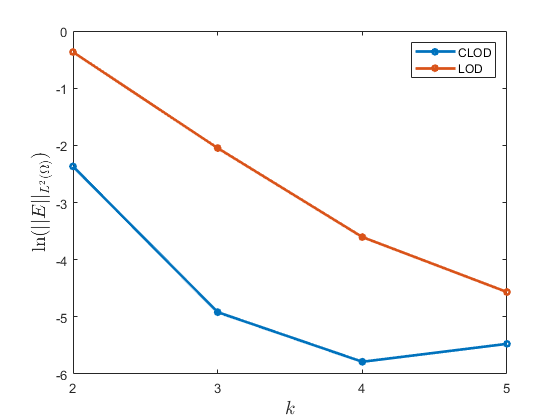}
		\caption{}
	\end{subfigure}
	
	\caption{Comparison of relative (a) energy errors and (b) $L^2$ errors of the ME-LOD and LOD in periodic microstructure.}
	\label{fig:1error}
\end{figure}

\subsection{Robustness of ME-LOD at high contrast ratios}
The last example shows the importance of the couple design of multiscale basis functions, which is designed to handle multiscale behavior in the coefficients. In this simulation, the body force $\bm{f}$ and heat source $g$ are chosen as

$$
\bm{f}(x, y)=\bm{0}, \quad g(x, y)=10 \times \exp \left(-\frac{(x-0.2)^2+(y-0.8)^2}{2 * 0.2^2}\right).
$$

The initial boundary condition is defined as

$$
\theta_0(x, y)=1000x(1-x)y(1-y).
$$

Then the Lam\'{e} coefficients $\mu, \lambda$, the conductivity coefficient $\kappa$ and the thermal expansion coefficient $\beta$ are shown to the left of Fig. \ref{fig:2coe},
\begin{figure}[htbp]
	\centering
	\includegraphics[scale=0.5]{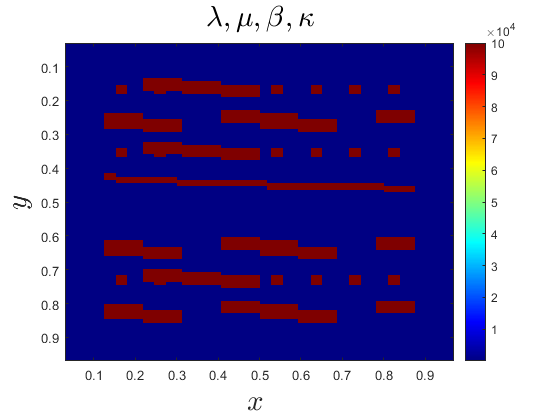}
	\caption{Contour plots of the coefficients in the third test. The contrasts of Lam\'{e} coefficients $\mu$, $\lambda$, thermal conductivity coefficient $\kappa$, and expansion coefficient $\beta$ are $10^5$}
	\label{fig:2coe}
\end{figure}
where the contrasts are chosen as $\lambda_{\max }: \lambda_{\min }= \mu_{\max }: \mu_{\min }= \kappa_{\max }: \kappa_{\min }= \beta_{\max }: \beta_{\min }= 10^1,10^2,\cdots ,10^8$ and the time step is $\tau=0.02$. Here, the $2^7 \times 2^7$ fine grid is used as a reference solution with the $2^4 \times 2^4$ coarse grid for the proposed ME-LOD and LOD. 	
The reference solutions ($\bm{u}^{\text {ref}}, \theta^{\text {ref}}$), the ME-LOD solutions ($\bm{u}^{melod}, \theta^{melod}$ ), and the LOD solutions ($\left.\bm{u}^{lod}, \theta^{lod}\right)$ of the test are presented in Fig. \ref{2sol}. It can be observed that the ME-LOD solutions have a higher accuracy than the LOD solutions when comparing the reference solutions, which is consistent with the periodic case. Moreover, to explore the influence of different coefficients on the results, we compare the energy errors and $L^2(\Omega)$ errors of the ME-LOD and LOD with the change of $\alpha$. Table \ref{tab:2error} reports the energy errors $ E_w$ and the $L^2(\Omega)$- errors of the ME-LOD and LOD for the third test in detail, where the contrast ratio of $\alpha_{\max}$ and $\alpha_{min}$ constantly varies. To make our point clearer, we also show the trend of the energy errors with $\alpha_{\max}$ and $\alpha_{\min}$ as contrast ratios in Fig. \ref{fig:2error}. For both the ME-LOD and LOD, we observe that the energy errors of the ME-LOD are almost at the same level, while those of the LOD change significantly. Although the whole system becomes extremely complex with increasing the contrast ratio or the variance of two material coefficients, the energy error of the whole system stays in a stable state, which is hardly affected by the complexity of the system.
Unlike the results of periodic micro-structures, the LOD method cannot find the correct numerical solution when the contrast of multiscale coefficients increases, even with an increase in patch size. In contrast, ME-LOD causes the relative error to increase slowly as the ratio increases, showing that our proposed method has stronger robustness for high-contrast cases.

\begin{figure}[htpb]
	
	\centering
	\begin{subfigure}[b]{1.0\textwidth}
		\includegraphics[width=0.32\textwidth]{2u1\_fem4}
		\hfill
		\includegraphics[width=0.32\textwidth]{2u2\_fem4}
		\hfill
		\includegraphics[width=0.32\textwidth]{2u3\_fem4}
		\caption{FEM}
	\end{subfigure}
	\hfill
	\begin{subfigure}[b]{0.95\textwidth}
		\includegraphics[width=0.32\textwidth]{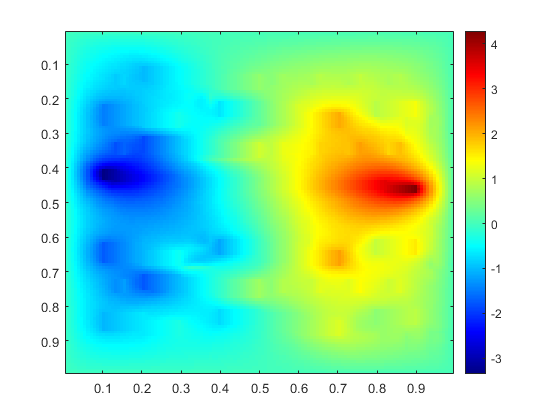}
		\hfill
		\includegraphics[width=0.32\textwidth]{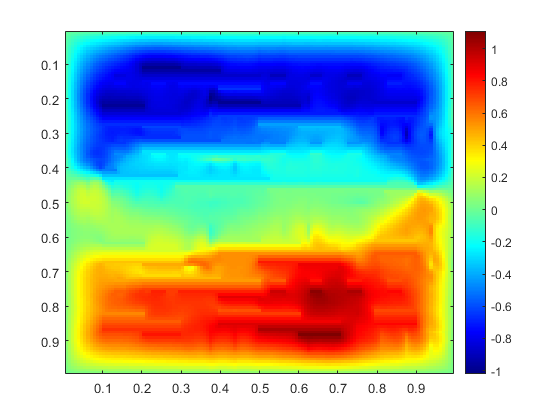}
		\hfill
		\includegraphics[width=0.32\textwidth]{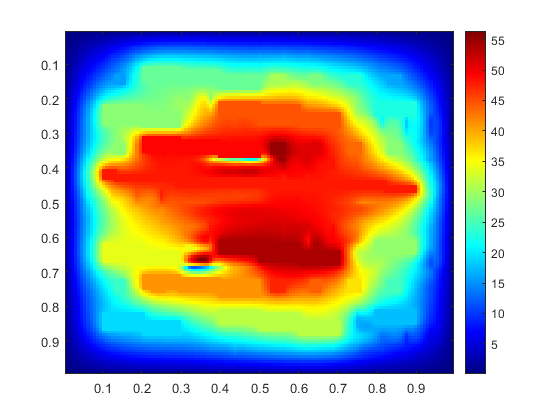}
		\caption{ME-LOD}
	\end{subfigure}
	\hfill
	\caption{Contour plots of the numerical solutions}
	\label{2sol}
\end{figure}

\begin{table}[htbp]
	\centering
	\caption{Relative errors with different ratios of multiscale coefficients.}
	\label{tab:2error}
	\begin{tabular}{cllll}
		\hline ratio & $\norm{E_w^{melod}}_e$ & $\norm{E_w^{lod}}_e$ & $\norm{E_w^{melod}}_{L^2}$ & $\norm{E_w^{lod}}_{L^2}$  \\
		\hline 
		$10^1$ & 2.26E-02 & 7.09E-02 & 1.08E-03 & 4.66E-03 \\
		$10^2$ & 3.60E-02 & 4.20E-01 & 4.41E-03 & 3.60E-01 \\
		$10^3$ & 1.14E-01 & 3.97E-01 & 2.77E-02 & 1.10E-01 \\
		$10^4$ & 1.26E-01 & 4.23E-01 & 5.74E-02 & 1.53E-01 \\
		$10^5$ & 6.55E-02 & 4.03E-01 & 4.98E-02 & 1.57E-01 \\
		$10^6$ & 4.13E-02 & 3.87E-01 & 5.91E-02 & 1.50E-01 \\
		$10^7$ & 4.55E-02 & 9.17E-01 & 8.33E-02 & 9.07E-01 \\
		$10^8$ & 6.81E-02 & 9.15E-01 & 1.14E-01 & 9.04E-01 \\
		\hline
	\end{tabular}
	
\end{table}
\begin{figure}[htpb]
	
	\centering
	
	\hfill
	\begin{subfigure}[b]{0.95\textwidth}
		\includegraphics[width=0.47\textwidth]{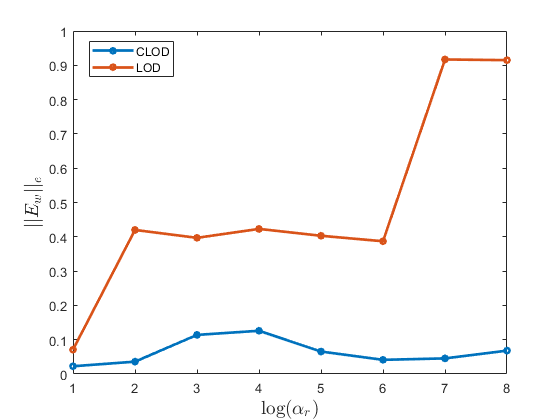}
		\hfill
		\includegraphics[width=0.47\textwidth]{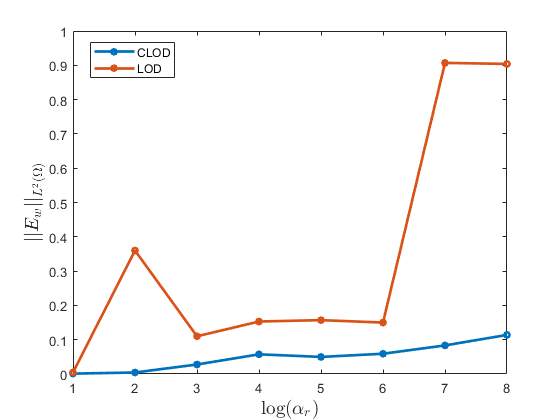}
	\end{subfigure}
	\hfill
	
	\caption{Comparison of relative energy errors for the ME-LOD and LOD with different contrast ratios of $\alpha$}
	\label{fig:2error}
\end{figure}

The first test shows that as the grid is refined, the ME-LOD method achieves the desired convergence rate, supporting its effectiveness in solving multiscale thermomechanical problems. The second test verifies that ME-LOD achieves higher precision for heterogeneous materials with periodic microstructures, and the third test further illustrates ME-LOD's enhanced robustness and stability in addressing complex thermomechanical problems involving high-contrast materials. Furthermore, the last two tests demonstrate that the newly constructed ME-LOD method significantly outperforms the traditional LOD in accuracy and efficiency. These results confirm the broad applicability and substantial advantages of ME-LOD in handling complex material challenges. Those numerical results demonstrate that the ME-LOD is computationally quite efficient and accurate with wide applicability in many scenarios.

\section{Conclusions}
\label{sec:conclusion}
This study has proposed a ME-LOD method that advances multiscale computation for thermomechanical problems. 
Unlike conventional approaches that treat thermal and mechanical interactions separately, ME-LOD constructs a unified operator that improves both the accuracy and efficiency of multiscale simulations. 
The new basis functions derived from ME-LOD capture localized effects more effectively, ensuring robust performance even under high material contrasts. 
Numerical tests demonstrate that ME-LOD consistently outperforms traditional LOD approaches in terms of accuracy and computational efficiency, establishing it as a powerful alternative for analyzing complex multiscale phenomena in heterogeneous materials. 
Although this study focuses on thermomechanical coupling---a representative case of multiphysics interactions---the ME-LOD framework is inherently versatile and can be readily extended to other coupled systems, such as electrothermal or magnetomechanical problems.

\bmhead{Acknowledgements}
The authors also thank the anonymous reviewers for their valuable comments and suggestions, which helped to improve the quality of this manuscript.

\bmhead{Author contribution}
Yuzhou Nan: writing, methodology, theoretical analysis. Yajun Wang: methodology, theoretical analysis. Changqing Ye: writing, review, editing. Hang Qi: review, editing. Xiaofei Guan: writing, methodology, review, editing.

\bmhead{Funding}This work was supported by the National Natural Science Foundation of China (No.~12271409), the Foundation of National Key Laboratory of Computational Physics, the Basic Research Project of CNPC (2023ZZ05), and the Fundamental Research Funds for the Central Universities.

\section*{Declarations}

\bmhead{Conflict of interest}The authors declare no competing interests.



\bibliography{sn-bibliography}

\end{document}